\documentclass{article}
\usepackage[utf8]{inputenc}
\usepackage{url}
\usepackage{authblk}

\usepackage{amsfonts}				
\usepackage{amsmath}				
\usepackage{siunitx}  				
\usepackage{xfrac}					
\usepackage{bm}						
\usepackage{amsthm}
\usepackage{siunitx}

\title{A constant lower bound for the union-closed sets conjecture}
\author{Justin Gilmer\thanks{gilmer@google.com}}
\affil{Google Research, Brain Team}

\begin{document}

\maketitle

\newcommand{\aintb}[2]{\frac{| #1 \cap #2 |}{|#1|}}
\newcommand{\e}{\mathop{\mathbb{E}}}
\newcommand{\f}{\mathcal{F}}
\newcommand{\bound}{\frac{3 - \sqrt{5}}{2}}
\newcommand{\weakbound}{0.1}
\newcommand{\weakboundp}{0.01}
\newcommand{\hu}[2]{H(#1 + #2 - #1 #2)}
\newcommand{\hup}{\hu{p_c}{p_{c'}}}
\newcommand{\czero}{\mathcal{C}_0}
\newcommand{\cone}{\mathcal{C}_1}
\newcommand{\constant}{1.26}

\newcommand{\x}{X}

\newtheorem{theorem}{Theorem}
\newtheorem{corollary}{Corollary}
\newtheorem{conjecture}{Conjecture}
\newtheorem{lemma}{Lemma}
\newtheorem{claim}{Claim}

\begin{abstract}
  We show that for any union-closed family $\f \subseteq 2^{[n]}, \f \neq \{\emptyset\}$, there exists an $i \in [n]$ which is contained in a $\weakboundp$ fraction of the sets in $\f$. This is the first known constant lower bound, and improves upon the $\Omega(\log_2(|\f|)^{-1})$ bounds of Knill and W\'{o}jick. Our result follows from an information theoretic strengthening of the conjecture. Specifically, we show that if $A, B$ are independent samples from a distribution over subsets of $[n]$ such that $Pr[i \in A] < \weakboundp$ for all $i$ and $H(A) > 0$, then $H(A \cup B) > H(A)$.
\end{abstract}

\section{Introduction}

\begin{sloppypar}

We study families of finite sets which are \emph{union-closed}. A family $\f \subseteq 2^{[n]}$ is said to be \emph{union-closed} if for every $A, B \in \f$ the set $A \cup B \in \f$. Frankl in 1979 \cite{frankl1995extremal} conjectured that any such family $\f \neq \{\emptyset\}$ should contain an \emph{abundant element}---that is an $i \in [n]$ which is contained in at least half of the sets in $\f$. Due to the simplicity of the problem statement, the union-closed conjecture has received substantial interest over the past 40 years, with over 50 publications proving special cases or providing reformulations of the problem \cite{bruhn2015journey}. The problem was also explored in Polymath11 \cite{polymath}, which considered several interesting strengthenings to the conjecture, some of which were shown to be false. The best prior bound which does not place additional assumptions on $\f$ is due to Knill \cite{knill1994graph} (with improvement by W\'{o}jick \cite{wojcik1999union}), who proves that there is an element contained in at least $\Omega(\frac{|\f|}{\log_2(|\f|)})$ sets. Some special cases are known which make strong assumptions on the family $\f$. For example Balla, Bollab\'{a}s, and Eccles \cite{balla2013union} show the conjecture holds when $|\f| \geq \frac{2}{3}2^n$. This was later improved by Karpas \cite{karpas2017two} under the assumption that $|\f| \geq 2^{n-1}$. We refer the interested reader to the survey of Bruhn and Schaudt \cite{bruhn2015journey} for an in depth survey of prior work on the problem. 

In this work, we prove the following theorem.

\begin{theorem}\label{thm:t1}
Let $A$ and $B$ denote independent samples from a distribution over subsets of $[n]$. Assume that for all $i \in [n]$, $Pr[i \in A] \leq \weakboundp$. Then $H(A \cup B) \geq  \constant H(A)$.
\end{theorem}

When $H(A) > 0$, Theorem~\ref{thm:t1} implies that $H(A \cup B) > H(A)$. Note that if we sample $A, B$ independently and uniformly at random from a union-closed family $\f$, then $H(A \cup B) \leq H(A)$. This follows because $A \cup B$ is a distribution over $\f$ and the entropy of a distribution over $\f$ is maximized when it is the uniform distribution. We obtain as an immediate corollary

\begin{theorem}\label{thm:t2}
Let $\f \subseteq 2^{[n]}$ be a union-closed family, $\f \neq \{\emptyset\}$. Then there exists $i \in [n]$ that is contained in at least a $\weakboundp$ fraction of the sets in $\f$.
\end{theorem}

We note that Theorem~\ref{thm:t1} operates in a more general setting than the union-closed conjecture as we allow $A$ to be sampled from an arbitrary probability distribution over a family $\f$. Consider the following illustrative examples.

{\bf Example 1:} Let $A = (A_1, A_2, \cdots, A_n)$ be a random subset of $[n]$ such that each $A_i$ are iid Bernoulli random variables with probability $p$. Then $H(A) = H(p)n$ and $H(A \cup B) = H(2p - p^2)n$.

{\bf Example 2:} Let $A = [n]$ with probability $p$ and $A = \emptyset$ with probability $1 - p$. Then  $H(A) = H(p)$ and $H(A \cup B) = H(2p - p^2)$.

In examples 1 and 2 the ratio $\frac{H(A \cup B)}{H(A)} = \frac{H(2p - p^2)}{H(p)}$. For these cases, when $p < \bound$, it follows that $H(A \cup B) > H(A)$. When $p=\bound$ then $H(A \cup B) = H(A)$ and when $p > \bound$ we get $H(A \cup B) < H(A)$. We hypothesize that these examples are extremal in the following sense: for any distribution $A$, if $Pr[A_i \leq p]$ for all $i$ then $H(A \cup B) \geq \frac{H(2p - p^2)}{H(p)} H(A)$.

The following example was useful in motivating some of the proof techniques we employ:

{\bf Example 3:} Sample $A \subseteq [n]$ in the following manner. First sample $A_1$ from a Bernoulli distribution with probability $p$. Then, conditioned on the event that $A_1 = 1$, sample each $A_i$ from iid Bernoulli distributions with probability $q = 0.99$. Otherwise, if $A_1 = 0$ then each $A_i = 0$. To calculate $H(A)$, we apply the chain rule to get $H(A) = H(A_1, A_{>1}) = H(A_1) + H(A_{> 1} | A_1)$. The conditional entropy can be computed as
\[H(A_{> 1} | A_1) = Pr[A_1 = 0] \cdot 0 + Pr[A_1 = 1]H(q)(n-1).\]

Thus $H(A) = H(p) + pH(q)(n-1)$. Via a similar calculation we get $H(A \cup B) = H(2p - p^2) + 2p(1-p)H(q)(n-1) + p^2 H(2q - q^2)(n-1)$.

In Example 3, for $n$ large and $p$ small, $H(A \cup B)$ is dominated by the term $2p(1-p)H(q)(n-1)$. This corresponds to the event that exactly one of $A_1, B_1$ is equal to $1$. It follows that $\frac{H(A \cup B)}{H(A)} \approx 2(1-p)$. Note in this case, the entropy $H(A \cup B | A_1 = B_1 = 1)$ is small relative to $H(A | A_1 = 1)$. We will discuss this example further in Section~\ref{sec:lemma}.

Examples 1 and 2 imply that if $Pr[A_i =1] \geq \bound$ then it is possible that $H(A \cup B) \leq H(A)$. Because $\bound < 0.5$, any stronger bound for Theorem~\ref{thm:t1} will not be sufficient to resolve the union-closed conjecture. In Section~\ref{sec:kl} we discuss a promising direction for additionally leveraging the assumption that $A$ is chosen uniformly over the family $\f$ which might improve the bound to 0.5.

\section{Notation and Preliminaries}
Throughout the paper we use $\log(x)$ to denote the base 2 logarithm of $x$. If $X, X'$ are Bernoulli random variables, we will use $X \cup X'$ to denote $\max(X, X')$. 

We quickly review two properties of conditional entropy that we require to complete the proofs. We refer the reader to Cover and Thomas \cite{cover1999elements} for additional background on information theory.

\begin{enumerate}
    \item {\bf Chain Rule for Entropy:} For a sequence of random variables $X_1, \cdots, X_n$, denote $X_{< i} = (X_1, \cdots, X_{i-1})$. Then $H(X_1, \cdots, X_n) = \sum\limits_i H(X_i | X_{<i})$.
    \item For random variables $X$ and $Y$ and a function $f(Y)$, \[H(X | Y) \leq H(X | f(Y)).\]
\end{enumerate}

We quickly prove property (2).
\begin{proof}
The sequence $X \rightarrow Y \rightarrow f(Y)$ forms a Markov chain. Thus by the data processing inequality:
\begin{align*}
    I(X: f(Y)) & \leq  I(X: Y) \\
    H(X) - H(X | f(Y)) & \leq  H(X) - H(X | Y) \\
    H(X | Y)  & \leq H(X | f(Y))
\end{align*}

\end{proof}

\section{Main Result}
In this section we prove our main result. We use $A_{<i} = (A_1, \cdots, A_{i-1})$ to denote the sequence of indicator random variables, where $A_i = 1$ if and only if $i \in A$. The proof strategy relies on revealing the bits of $A \cup B$ and $A$ one at a time and showing at each step that 
\begin{equation} \label{eq:main}
    H((A \cup B)_i | (A \cup B)_{<i}) \geq \constant H(A_i | A_{<i}).
\end{equation} 
By applying the chain rule this will imply that $H(A \cup B) \geq \constant H(A)$. 

The proof of equation (\ref{eq:main}) will rely on this key technical lemma, the proof of which is provided in Section~\ref{sec:lemma}.

\begin{lemma} \label{lem:l5}
Let $C$ denote a random variable over a finite set $S$. For each $c \in S$, let $p_c$ be a real number in $[0, 1]$. Let $\x$ be a Bernoulli random variable sampled according to the following process: first sample $c \sim C$, then sample $\x$ with  $Pr[\x = 1| C = c] = p_c$. Assume further that $\mathbb{E}[\x] \leq 0.01$. Let $C'$ be an iid copy of $C$, and sample $\x'$ conditioned on $C'$ according to the same process (so $Pr[\x' = 1| C' = c] = p_c$, and $\x'$ is independent of $\x$ and $C$). Then 
\[
    H(\x \cup \x' | C, C') \geq \constant H(\x | C).
\]
\end{lemma}

We note that Lemma~\ref{lem:l5} can be restated a bit more succinctly that assuming  $\{p_c\}_{c \in S} \subset [0,1]$ is a finite sequence of real numbers satisfying $\mathbb{E}_c[p_c] \leq 0.01$, then:

\[
    \mathbb{E}_{c, c'}\left[\hu{p_c}{p_{c'}}\right] \geq \constant \mathbb{E}_{c} \left[H(p_c)\right].
\]

Here, $\x, \x'$ correspond to the random bits $A_i, B_i$ respectively, and $C, C'$ correspond to the histories $A_{<i}, B_{<i}$. The constant $0.01$ was not optimized as new ideas will be needed to achieve a tight result. We hypothesize that if $\mathbb{E}[\x] < \bound$, then  $H(\x \cup \x' | C, C') \geq (1 + \epsilon) H(\x | C)$  for an $\epsilon > 0$ which depends on the value of $\mathbb{E}[\x]$. We discuss challenges in obtaining a stronger bound for Lemma~\ref{lem:l5} along with counter examples to natural strengthenings in Section~\ref{sec:lemma}. 

Assuming Lemma 1, we now prove our main result:

\begingroup
\def\thetheorem{\ref{thm:t1}}
\begin{theorem}
Let $A, B$ be independent samples from a distribution over subsets of $[n]$ such that $Pr[i \in A] \leq 0.01$ for all $i$. Then $H(A \cup B) \geq \constant H(A)$.
\end{theorem}
\addtocounter{theorem}{-1}
\endgroup

\begin{proof}

We first show for all $i$, 
\[H((A \cup B)_i | (A \cup B)_{< i}) \geq \constant H(A_i | A_{< i}).\]

By applying property (2) of conditional entropy we get
\begin{equation}\label{eq:cond_key}
    H((A \cup B)_i | (A \cup B)_{< i}) \geq H((A \cup B)_i | A_{< i}, B_{<i}).
\end{equation}

We pause here to remark that (\ref{eq:cond_key}) is the crucial step which takes advantage of the power of the information theoretic formulation. Because $(A \cup B)_{<i}$ is simply a function of $A_{< i}, B_{<i}$, the entropy in $(A \cup B)_i$ can not increase if we additionally assume we know the full history of $A_{<i}, B_{<i}$. Conditioning on $A_{< i}, B_{< i}$ dramatically simplifies the analysis, as these are iid. Additionally, $A_i$ and $B_i$ are Bernoulli random variables whose distribution are determined by the sampled values of $A_{<i}$ and $B_{<i}$ respectively. Thus by Lemma~\ref{lem:l5} we conclude that
\begin{equation}
    H((A \cup B)_i | A_{< i}, B_{< i}) \geq \constant H(A_i | A_{< i}).
\end{equation}

To end the proof we repeatedly apply the chain rule to conclude that
\begin{equation}
    H(A \cup B) \geq \constant H(A).
\end{equation}
\end{proof}

\section{Proof of Lemma 1}\label{sec:lemma}

For this section, we can forget all of the structure contained in the random variables $A_{<i}$ and $B_{<i}$. Lemma 1 only assumes that they are iid over some finite set $S$. Recall that Lemma 1 can be stated as
\begin{equation} \label{eq:ineq}
    \mathbb{E}_{c, c'}\left[\hu{p_c}{p_{c'}}\right] \geq \constant \mathbb{E}_{c} \left[H(p_c)\right]
\end{equation}
under the assumption that $\mathbb{E}_c[p_c] \leq 0.01 = \mu$.

A natural approach to Lemma 1 is to try to apply Jensen's inequality to the function $f(p_c, p_{c'}) = \hu{p_c}{p_{c'}} - H(p_c)$. However, this $f$ is not convex in $p_{c'}$. Additionally, it does not hold in general that $\mathbb{E}_{c, c'} \left[\hu{p_c}{p_{c'}} - H(p_c)\right] \geq H(2\mu - \mu^2) - H(\mu)$. For example, conditioned on $C$ there may be no entropy left in $\x$, in which case the left hand side is 0! This is exactly what will happen in Example 2 discussed in the introduction---after revealing the first bit $A_1$, all subsequent bits become deterministic. This example demonstrates that some natural symmetrizations such as $g(p_c, p_{c'}) = \hu{p_c}{p_{c'}} - \frac{H(p_c) + H(p_{c'})}{2}$  are not convex.

Another natural approach is to look for a purely information theoretic proof of Lemma 1. Indeed, one hypothesis is that there is nothing special about the union function here, but for any function $f$, $H(f(\x, \x') | C, C') \geq H(\x | C)$ whenever $H(f(\x, \x')) \geq H(\x)$. However, this strengthening turns out to be false. Consider the case where both $\x$ and $C$ are uniform over the set $\{0, 1, 2, 3\}$. Furthermore, let $\x|C$ be uniform over $\{0, 2\}$ when $C \in \{0, 2\}$, and $\x|C$ be uniform over  $\{1, 3\}$ when $C \in \{1, 3\}$. Finally, define $f(x, x') = (x \mod 2, x' \mod 2)$. Then $H(f(\x, \x')) = H(\x) = \log(4)$, $H(\x | C) = 1$, but $H(f(\x, \x') | C, C') = 0$. Thus any proof of Lemma~\ref{lem:l5} will need to make careful use of properties of the union function. 

Having been unable to make the above two proof strategies work, we resort to a more direct estimation of the terms in inequality (\ref{eq:ineq}). Our argument is quite wasteful and surely is far from tight. First we provide a proof sketch. We let $\czero = \{c | p_c \leq 0.1\}$ and let $\cone = \czero^c$. 
 
 Using the assumption that $E[\x] \leq 0.01$ we apply Markov's inequality to get that
\begin{equation} \label{eq:eq3}
    Pr[c \in \cone] = Pr[p_c > 0.1] \leq \frac{E_{c}[p_c]}{0.1} \le 0.1
\end{equation}

This implies that $Pr[C_0] \geq 0.9$. In what follows we will sometimes write $\czero$ as shorthand for the event that $C \in \czero$. Similarly $\czero'$ refers to the event that $C' \in \czero$. For example, the conditional entropy $H(X |C)$ can be written as
 \[H(X | C) = Pr[\czero]H(X | \czero) + Pr[\cone]H(X | \cone).\]
 
 We first note that conditioned on the event that both $C, C' \in C_0$, the entropy $H(\x \cup \x')$ will be a constant factor larger than $\frac{H(\x) + H(\x')}{2}$. This can be leveraged to prove that
\begin{equation}\label{eq:eq00}
    Pr[\czero]^2H(\x \cup \x' | \czero, \czero') \geq \constant Pr[\czero]H(\x | \czero).
\end{equation} 

Then, in the event that exactly one of $c, c' \in \czero$ we can show that $H(\x \cup \x') \geq 0.9H(\x)$. Using this property, we will show that

\begin{equation}\label{eq:eq01}
    2Pr[\czero]Pr[\cone]H(\x \cup \x' | \czero, \cone') \geq 1.62 Pr[\cone]H(\x |  \cone).
\end{equation} 

Example 3 discussed in the introduction helped to motivate the decomposition considered in equations (\ref{eq:eq00}) and (\ref{eq:eq01}). In this example, most of the entropy in $H(\x | C)$ comes from the event that $A_1 = 1$ (this corresponds to the event $\cone$). This entropy is dominated by the corresponding event that exactly one of $A_1$ and $B_1$ are equal to 1, which is exactly the conclusion of equation (\ref{eq:eq01}). This example also demonstrates that entropy coming from the term $Pr[\cone]^2 H(\x, \x' | C, C' \in \cone)$ may be small relative to $Pr[\cone]H(\x |\cone)$. In this work we throw this term away, it is non-negative and the sum of the left hand side of (\ref{eq:eq00}) and (\ref{eq:eq01}) are already larger than $H(\x|C)$. However, a tight version of Lemma~\ref{lem:l5} will require a more careful analysis.

We now make the above proof sketch rigorous with the following sequence of lemmas.

\begin{lemma} \label{lem:l1}
    Assume $p, p' \leq \weakbound$. Then $H(p + p' - p p') \geq 1.4 \left( \frac{H(p) + H(p')}{2}\right)$.
\end{lemma}
\begin{proof}
    Note the lemma holds when $p = p' = 0$. We let $D = [0, 0.1] \times [0, 0.1] - \{(0, 0)\}$. Figure~\ref{fig:entropy_ratio} plots the function $f(p, p') = \frac{2 H(p + p' - p p')}{H(p) + H(p')}$ for $(p, p') \in D$ where the lemma can be checked visually. More formally, by concavity of $H$, $\frac{H(p) + H(p')}{2} \leq H\left(\frac{p + p'}{2}\right)$. Additionally, when $0 \leq p, p' \leq 0.1$, we have $p + p' - p p' \geq 0.9(p + p')$. Thus in the given domain, $f(p, p') \geq \frac{H(0.9(p + p'))}{H(0.5(p + p'))}$. The function $g(p) = \frac{H(0.9p)}{H(0.5 p)}$ for $p \in (0, 0.2]$ is minimized at $p = 0.2$. This implies that over the domain, $f(p, p') > g(0.2) = 1.45$.
\end{proof}

\begin{lemma} \label{lem:l2}
  For any $p, p' \in [0, 1]$, $H(p + p' - p p') \geq (1-p)H(p')$.
\end{lemma}
\begin{proof}
    By concavity of $H$, \[H(p \cdot 1 + (1 - p)p') \geq p H(1) + (1 - p) H(p') = (1-p)H(p').\]
\end{proof}

 For the next lemmas, we use $q$ to denote the distribution of $C$, that is $q(c) = Pr[C = c]$. Additionally $q_{0}$ denotes the distribution of $C$ conditioned on the event that $C \in \czero$. So for $c \in \czero$, $q_{0}(c) = \frac{q(c)}{Pr[C \in \czero]}$.

\begin{figure}
    \includegraphics[width=10cm]{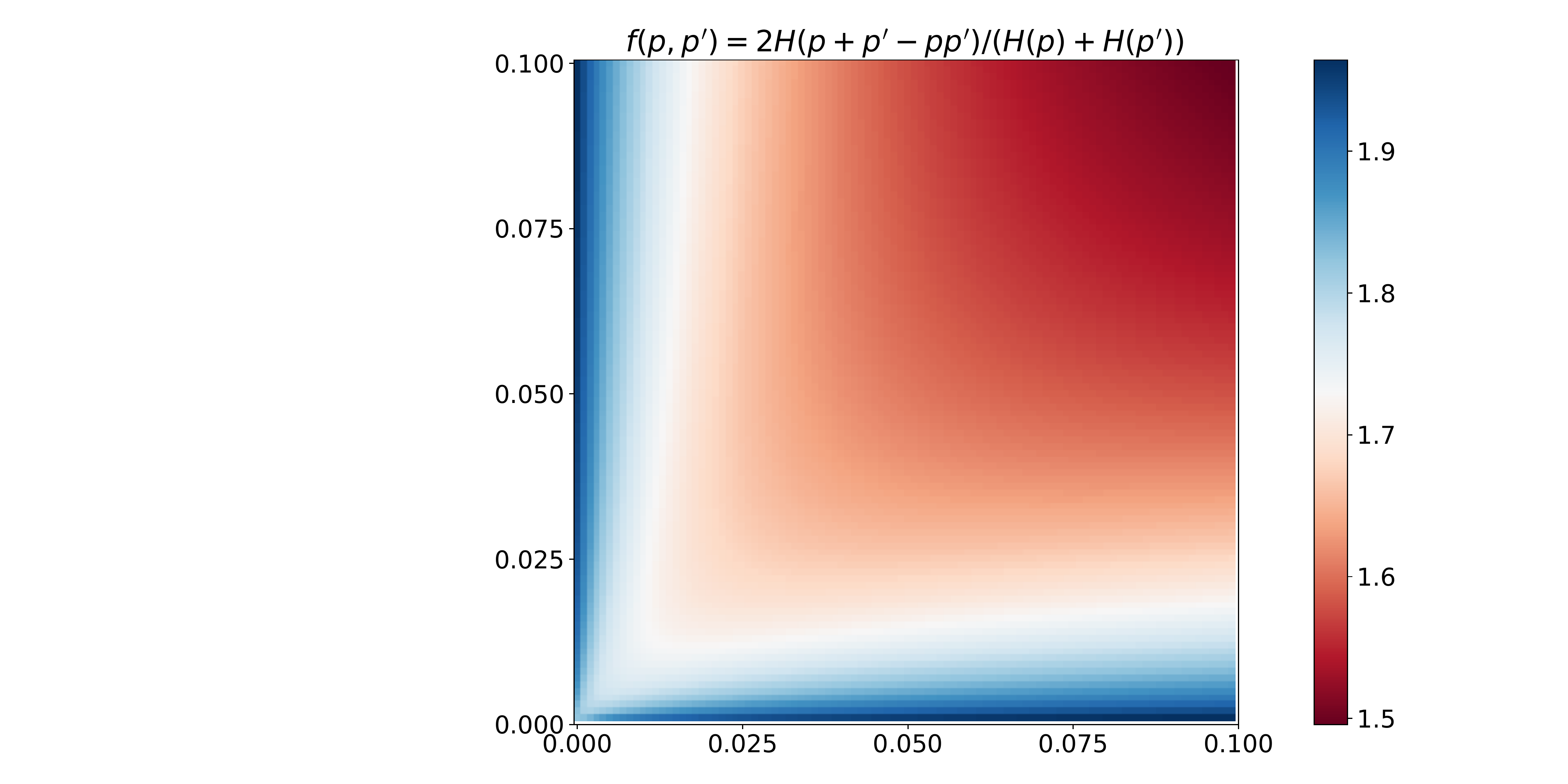}
    \caption{Plotting the function $f(p, p') = \frac{2H(p + p' - p p')}{H(p) + H(p')}$ over $0 \leq p, p' \leq 0.1$. The minimum value of $1.496$ is achieved at $p = p' = 0.1$.}
    \label{fig:entropy_ratio}
\end{figure}

\begin{lemma} \label{lem:l3}
Under the assumption that $\mathbb{E}[\x] \leq 0.01$,
\[Pr[\czero]^2  H(\x \cup \x' | \czero, \czero') \geq \constant Pr[\czero] H(\x | C \in \czero)\]
\end{lemma}
\begin{proof}

\begin{align*}
    Pr[\czero] H(\x | C \in \czero) & = Pr[\czero] \mathbb{E}_{c \sim q_0} H(p_c) \\
    & = \frac{Pr[\czero]}{2} \mathbb{E}_{c \sim q_0}\left[ H(p_c) + \mathbb{E}_{c' \sim q_0} H(p_{c'}) \right] \\
    & = Pr[\czero] \mathbb{E}_{c, c' \sim q_0}\left[  \frac{H(p_c) + H(p_{c'})}{2} \right] \\
 \text{(By Lemma~\ref{lem:l1}) \qquad \qquad}  & \leq \frac{Pr[\czero]}{1.4} \left[ \mathbb{E}_{c, c' \sim q_0} \hu{p_c}{p_{c'}} \right] \\
  \text{($Pr[\czero] \geq 0.9$)  \qquad \qquad}  & \leq \frac{Pr[\czero]^2}{\constant} H(\x \cup \x' | C, C' \in \czero)
\end{align*}

Multiplying both sides by $\constant$ yields the desired result.

\end{proof}

\begin{lemma}\label{lem:l4}
Under the assumption that $\mathbb{E}[\x] \leq 0.01$,
 \[2 Pr[\czero, \cone'] H(\x \cup \x' | \czero, \cone') \geq 1.62 Pr[\cone] H(\x | C \in \cone)\]
\end{lemma}
\begin{proof}

\begin{align*}
    2 Pr[\czero, \cone'] H(\x \cup \x' | \czero, \cone') & = 2 \sum\limits_{c \in \czero, c' \in \cone} q(c) q(c') \hu{p_c}{p_{c'}} \\
   \text{(by Lemma~\ref{lem:l2}) \qquad \qquad}  & \geq 2 \sum\limits_{c \in \czero, c' \in \cone} q(c) q(c') (1 - p_c) H(p_{c'}) \\
   & = 2\sum\limits_{c \in \czero} q(c)(1 - p_c) \left[\sum\limits_{c' \in \cone} q(c') H(p_{c'})\right] \\
    & = 2 Pr[\cone'] H(\x' | \cone') \sum_{c \in \czero} q(c)(1 - p_c)  \\
 \text{(using $p_c \leq 0.1$) \qquad \qquad}    & \geq 2 Pr[\cone']H(\x | \cone') \sum_{c \in \czero} q(c) 0.9 \\
    & = 1.8 Pr[\czero] Pr[\cone'] H(\x' | \cone') \\
  \text{(using $Pr[\czero] \geq 0.9$) \qquad \qquad}   & \geq 1.62 Pr[\cone'] H(\x' | \cone')
\end{align*}
\end{proof}

We can now quickly finish the proof of Lemma~\ref{lem:l5}.

\begin{proof}

To show that $H(\x \cup \x' | C, C') \geq 1.26 H(\x | C)$, we write $H(\x \cup \x' | C, C')$ as a sum of three disjoint events:

\begin{enumerate}
    \item $Pr[C, C' \in \czero] H(\x \cup \x'| C, C' \in \czero)$
    \item $2Pr[C \in \czero] Pr[C' \in \cone] H(\x \cup \x' | C \in \czero, C' \in \cone)$
    \item $Pr[C, C' \in \cone] H(\x \cup \x'| C, C' \in \cone)$
\end{enumerate}

By Lemma~\ref{lem:l3}, event (1) has higher entropy than $ \constant Pr[C \in \czero] H(\x | C \in \czero)$. By Lemma~\ref{lem:l4}, event (2) has higher entropy than $1.62 Pr[C \in \cone]H(\x | C \in \cone)$. Finally, event (3) has non-negative entropy. Thus $H(\x \cup \x' | C, C') \geq \constant  H(\x | C)$.

\end{proof}

\section{A possible path towards resolving the conjecture} \label{sec:kl}

It is clear that there is more ground to be covered with the information theoretic approach we have initiated in this work. A tight version of Lemma~\ref{lem:l5} would imply a $\bound$ lower bound on the maximum element frequency for union-closed families. Because $\bound < \frac{1}{2}$, additional ideas will be needed to resolve union-closed conjecture. In this section we discuss a potential direction towards this strengthening.

In cases where $p$ is close to $\frac{1}{2}$, the distribution of $A \cup B$ seems to be far from uniform. Thus it may still hold that $|\f \cup \f| > |\f|$\footnote{We use $\f \cup \f$ to denote $\{A \cup B | A, B \in \f\}$.} even though $H(A \cup B) \leq H(A)$. To quantify how far from uniform the distribution $A \cup B$ is, it is useful to consider the KL-divergence $D(A \cup B || A)$. When $A$ is the uniform distribution over a union-closed family $\f$, it holds that\footnote{See \cite{cover1999elements} Theorem 2.6.4.}
\begin{equation} \label{eq:kl}
    D(A \cup B || A) + H(A \cup B) = H(A) = \log( |\f|).
\end{equation}

We can study the quantity $D(A \cup B || A) + H(A \cup B)$ for more general distributions $A$---say if $A$ is not the uniform distribution, or $\f$ is not union-closed. For example, if $A$ denotes a single bit with probability $p$ of being 1, then when $p = 0.5$ it holds exactly that  $D(A \cup B || A) + H(A \cup B)$ = $H(A) = 1.0$. However, if $p < 0.5$ it holds that
\begin{equation} \label{eq:conjeq}
    D(A \cup B || A) + H(A \cup B) > H(A).
\end{equation}
If equation~(\ref{eq:conjeq}) ever holds for a distribution $A$, we can conclude that either $A$ is not the uniform distribution over $\f$ or the distribution $A \cup B$ has support outside of $\f$. 

Thus the union-closed sets conjecture would follow from showing the following:

\begin{conjecture} \label{conj1}
Let $A, B$ be iid samples from a distribution over a family of subsets of $[n]$. Assume that $Pr[i \in A] < 0.5$ for all $i$, and $H(A) > 0$. Then $H(A \cup B) + D(A \cup B || A) > H(A)$.
\end{conjecture}

\end{sloppypar}

\section{Conclusion}

We have established the first constant lower bound for the union-closed conjecture by studying the entropy of the union of two iid samples from a family $\f$. The methods presented are strong enough to derive the stronger conclusion that $H(A \cup B) \geq C_p H(a)$ for a constant $C_p > 0$ which depends on $p = \max\limits_i{Pr[A_i = 1]}$. However, we certainly have not derived the strongest possible bound $C_p$. We are hopeful that the approach initiated in this work will lead to a proof of the conjecture. Beyond proving the union-closed conjecture, the following questions could be interesting to consider

\begin{enumerate}

\item Does it hold for any distribution $A$ with $Pr[A_i = 1] \leq p$ for all $i$ that $H(A \cup B) \geq \frac{H(2p - p^2)}{H(p)} H(A)$?
\item Does Conjecture 1 hold?
\item Under what other assumptions on the distributions $A, B$ does it hold that $H(A \cup B) > H(A)$? Suppose for example that for fixed $k$ it holds that for every $X \in \binom{[n]}{k}$, $Pr[X \subseteq A] < p$. How small does $p$ need to be to conclude that $H(A \cup B) > H(A)$?
\end{enumerate}

{\bf Update (11/27/2022)} Shortly after publication of this preprint, three publications appeared which all prove tight versions of our Lemma 1 \cite{chase2022approximate, sawin2022improved, alweiss2022improved}. These results improve the resulting bound on Frankl's conjecture to $\bound \approx .38$. Sawin \cite{sawin2022improved} confirm Question 1 when $p \leq \bound$. However, when $p > \bound$ it only holds that $H(A \cup B) \geq (1 - p)\frac{2}{\sqrt{5} - 1}$. Sawin \cite{sawin2022improved} and Ellis \cite{ellis2022note} provide constructions refuting Conjecture~\ref{conj1}. It is noteworthy that Sawin's construction demonstrates that, without placing additional assumptions on the distribution $A$, incorporating the KL term cannot improve the resulting bound on Frankl's conjecture. 

\section*{Acknowledgement}

The author is grateful to Michael Saks and Swastik Kopparty for enlightening discussions and for reviewing initial versions of this work. Additionally, the author thanks Phil Long for his careful reading and feedback on the manuscript.

\bibliographystyle{plain}
\bibliography{bibliography}

\end{document}